\documentclass[12pt]{amsart}
\usepackage{fullpage}
\usepackage{amsfonts,amscd}
\usepackage{amssymb}
\usepackage{url}
\usepackage{graphicx}
\usepackage[english]{babel}

\usepackage{pstricks}
\usepackage{pst-plot}
\usepackage{pst-node}
\SpecialCoor

\theoremstyle{plain}
\newtheorem{theorem}                {Theorem}      [section]
\newtheorem{proposition}  [theorem]  {Proposition}

\newtheorem{lemma}        [theorem]  {Lemma}

\theoremstyle{definition}

\newtheorem{remark}       [theorem]  {Remark}

\numberwithin{equation}{section}

\def \R{{\mathbb R}}
\def \s{{\mathbb S}}

\usepackage{color}

\DeclareMathOperator{\cst}{constant}
\DeclareMathOperator{\grad}{grad}
\DeclareMathOperator{\trace}{trace}
\DeclareMathOperator{\Div}{div}
\DeclareMathOperator{\ricci}{Ricci}

\numberwithin{equation}{section}

\begin{document}

\title[]{Proper Biconservative immersions into the Euclidean space}

\author{S.~Montaldo}
\address{Universit\`a degli Studi di Cagliari\\
Dipartimento di Matematica e Informatica\\
Via Ospedale 72\\
09124 Cagliari, Italia}
\email{montaldo@unica.it}

\author{C.~Oniciuc}
\address{Faculty of Mathematics\\ ``Al.I. Cuza'' University of Iasi\\
Bd. Carol I no. 11 \\
700506 Iasi, ROMANIA}
\email{oniciucc@uaic.ro}

\author{A.~Ratto}
\address{Universit\`a degli Studi di Cagliari\\
Dipartimento di Matematica e Informatica\\
Viale Merello 93\\
09123 Cagliari, Italia}
\email{rattoa@unica.it}

\begin{abstract}
In this paper,  using the framework of equivariant differential geometry, we study proper $SO(p+1) \times SO(q+1)$-invariant biconservative hypersurfaces into the Euclidean space $\R^n$ ($n=p+q+2$) and proper $SO(p+1)$-invariant biconservative hypersurfaces into the Euclidean space $\R^n$ ($n=p+2$).
Moreover, we show that, in these two classes of invariant families, there exists no proper biharmonic immersion.
\end{abstract}



\thanks{Work supported by: PRIN 2010/11 -- Variet\`a reali e complesse: geometria, topologia e analisi armonica N. 2010NNBZ78 003 -- Italy; GNSAGA -- INdAM, Italy;
Romanian National Authority for Scientific Research, CNCS -- UEFISCDI, project number PN-II-RU-TE-2011-3-0108}

\maketitle

\section{Introduction}\label{intro}

A hypersurface $M^{n-1}$ in an $n$-dimensional Riemannian manifold $N^{n}$ is called {\it biconservative} if
\begin{equation}\label{eq-def-biconservative-space}
2 A(\grad f)+ f \grad f=2 f \ricci^N(\eta)^{\top}\,,
\end{equation}
where $A$ is the shape operator, $f=\trace A$ is $(n-1)$ times the mean curvature function and $\ricci^N(\eta)^{\top}$ is the tangent  component of the Ricci curvature of $N$ in the direction of the unit normal $\eta$ of $M$ in $N$.\\

As we shall detail in Section~\ref{sec:stress-energy-tensor}, biconservative  hypersurfaces are those with divergence-free stress-bienergy tensor and can be characterized as  the hypersurfaces with vanishing  tangent component of the bitension field
\begin{equation}\label{bitensionfield}
 \tau_2(\varphi) = - \Delta \tau(\varphi)- \trace R^N(d \varphi, \tau(\varphi)) d \varphi  \,\, .
\end{equation}

To give sense to \eqref{bitensionfield}, we recall that a smooth map $\varphi:(M,g)\to(N,h)$ is a {\it harmonic map}  if it is a critical point of the {\em energy} functional
\begin{equation}\label{energia}
E(\varphi)=\frac{1}{2}\int_{M}\,|d\varphi|^2\,dv_g \,\, ,
\end{equation}
of which the Euler-Lagrange equation is $\tau(\varphi)={\trace} \, \nabla d \varphi =0$.
A natural generalization of harmonic maps are the so-called {\it biharmonic maps}: these maps are the critical points of the bienergy functional (as suggested by Eells--Lemaire \cite{EL83})
\begin{equation}\label{bienergia}
    E_2(\varphi)=\frac{1}{2}\int_{M}\,|\tau (\varphi)|^2\,dv_g \,\, .
\end{equation}
In \cite{Jiang} G.~Jiang showed that the Euler-Lagrange equation associated to $E_2(\varphi)$ is given by $\tau_2(\varphi) =0$.\\

An immersed sub\-mani\-fold into a Riemannian manifold $(N,h)$ is called a {\it biharmonic sub\-mani\-fold} if the immersion is a biharmonic map.
Thus biharmonic hypersurfaces are biconservative. \\

In this paper we consider biconservative hypersurfaces in the Euclidean space $\R^n$. In this case \eqref{eq-def-biconservative-space} becomes
\begin{equation}\label{eq-def-biconservative-space-form}
2 A(\grad f)+ f \grad f=0\,.
\end{equation}
From \eqref{eq-def-biconservative-space-form} we see immediately that CMC hypersurfaces are biconservative. Thus our inte\-rest will be
on biconservative hypersurfaces which are not CMC: we shall call them \emph{proper} biconservative.

In \cite{CMOP2013} and \cite{HasVla95} the authors have classified proper biconservative surfaces in $\R^3$ proving that they must be of revolution.
In higher dimensional Euclidean spaces the situation is rather different, as shown in \cite{HasVla95}, where the authors have found other families
of biconservative hypersurfaces in $\R^4$ and, in particular, they have shown that some of them are $SO(1)\times SO(1)$-invariant. Recently, in \cite{Chen-munteanu}, the authors proved that a $\delta(2)$-ideal biconservative hypersurface in Euclidean space $\R^n$ ($n \geq 3$) is either minimal or a spherical hypercylinder. Moreover, we would like to mention that there is parallel study of biconservative hypersurfaces in semi-Riemannian geometry  (see,
for example, \cite{Fu}). \\

Our goal is to give a detailed description of biconservative $SO(p+1)\times SO(q+1)$-invariant hypersurfaces in $\R^n$, $n=p+q+2$, and biconservative $SO(p+1)$-invariant hypersurfaces in $\R^n$, $n=p+2$, using the framework of equivariant differential geometry in the spirit of \cite{ER}, \cite{Hsiang}  and
 \cite{MR}. Our analysis will lead us to the following main results (Theorem \ref{teoremabiconservativeshort} is an immediate consequence of Theorem \ref{teoremabiconservative} which we shall prove in section \ref{section-biconservative-existence}):\\

\begin{theorem}\label{teoremabiconservativeshort}
There exists an infinite family of proper $SO(p+1)
\times SO(q+1)$-invariant biconservative hypersurfaces (cones) in $\R^n$ ($n=p+q+2$). Their corresponding profile curves $\gamma(s)$  tend asymptotically to the profile of a minimal cone. If $p+q \leq 17$, at infinity the profile curves $\gamma$ intersect the profile of the minimal cone at infinitely many points, while, if $p+q \geq 18$,  at infinity the profile curves $\gamma$ do not intersect the profile of the minimal cone.
None of these hypersurfaces is complete.
\end{theorem}
\begin{remark} The lack of completeness is due to the fact that these hypersurfaces present a topological (cone-like) singularity at the origin of $\R^n$ .  
\end{remark}

 \noindent {\bf Theorem}~{\bf \ref{teo-revolution-catenary}.}
{\em There exists an infinite family of complete, proper $SO(p+1)$-invariant biconservative hypersurfaces in $\R^n$ ($n=p+2$).
Their corresponding profile curves $\gamma(s)$ are of ``{catenary}'' type.}\\

\vspace{2mm}

The study of biconservative hypersurfaces in the Euclidean space is also relevant for the study of biharmonic hypersurfaces.
In fact, for biharmonic submanifolds in $\R^n$,  it is still open the Chen's conjecture (see \cite{Chen}): {\it biharmonic submanifolds into $\R^n$ are minimal}.
The conjecture is still open even for biharmonic hypersurfaces in $\R^n$.
As we have already noticed, biharmonic hypersurfaces are biconservative, thus a way to tackle the Chen conjecture is to prove that, amongst the
proper biconservative hypersurfaces, none is proper biharmonic. Clearly, $SO(p+1)$-invariant hypersurfaces in $\R^n$ ($n=p+2$) have at most two distinct principal curvatures and, by a result of Dimitric (see \cite{dimitric}), any biharmonic hypersurface in $\R^n$ with at most two distinct principal curvatures is minimal.
By contrast, the $SO(p+1)\times SO(q+1)$-invariant hypersurfaces of Theorem \ref{teoremabiconservativeshort} have three distinct principal curvatures and there is
no general result that forces a biharmonic hypersurface in $\R^n$ with at least three distinct principal curvatures to be minimal. The only exception is
when the hypersurface is in $\R^4$, in which case it was proved in \cite{HasVla95} and \cite{Defever} that bihamonicity implies minimality.

Following this venue we show in Section~\ref{biharmonic-immersion-non-existence} that, amongst our proper $SO(p+1)
\times SO(q+1)$-invariant biconservative hypersurfaces in $\R^n$ ($n=p+q+2$), there are no proper biharmonic hypersurfaces. This result may be considered as a further  step towards the proof of Chen's conjecture.

\section{Biharmonic maps and the stress-energy tensor}\label{sec:stress-energy-tensor}

As described by Hilbert in~\cite{DH}, the {\it stress-energy}
tensor associated to a variational problem is a symmetric
$2$-covariant tensor $S$ conservative at critical points,
i.e. with $\Div S=0$.

In the context of harmonic maps $\varphi:(M,g)\to (N,h)$ between two Riemannian manifolds, the stress-energy tensor was studied in detail by
Baird and Eells in~\cite{PBJE} (see also \cite{AS} and \cite{BR}). Indeed, the Euler-Lagrange
equation associated to the energy functional \eqref{energia} is equivalent to the vanishing of the tension
field $\tau(\varphi)=\trace\nabla d\varphi$ (see \cite{ES}), and the tensor
$$
S=\frac{1}{2}\vert d\varphi\vert^2 g - \varphi^{\ast}h
$$
satisfies $\Div S=-\langle\tau(\varphi),d\varphi\rangle$. Therefore, $\Div S=0$ when the map is harmonic.

\begin{remark}\label{remark:conservative}
We point out that, in the case of isometric immersions, the condition $\Div S=0$ is always satisfied,
 since $\tau(\varphi)$ is normal.
\end{remark}

Now, we begin our study of the bienergy functional \eqref{bienergia} and of its associated Euler-Lagrange equation \eqref{bitensionfield}. In particular, we point out that, in the expression \eqref{bitensionfield} of the bitension field,
$\Delta$ is the rough Laplacian on sections of $\varphi^{-1} \, (TN)$ that, for a local orthonormal frame $\{e_i\}_{i=1}^m$ on $M$, is defined by
$$
    \Delta=-\sum_{i=1}^m\{\nabla^{\varphi}_{e_i}
    \nabla^{\varphi}_{e_i}-\nabla^{\varphi}_
    {\nabla^{M}_{e_i}e_i}\}\,\,.
$$
The curvature operator on $(N,h)$, which also appears in \eqref{bitensionfield}, can be computed by means of
$$
    R^N (X,Y)= \nabla_X \nabla_Y - \nabla_Y \nabla_X -\nabla_{[X,Y]} \,\, .
$$

The study of the stress-energy tensor for the
bienergy was initiated  in \cite{GYJ2} and afterwards developed in \cite{LMO}. Its expression is \begin{eqnarray*}
S_2(X,Y)&=&\frac{1}{2}\vert\tau(\varphi)\vert^2\langle X,Y\rangle+
\langle d\varphi,\nabla\tau(\varphi)\rangle \langle X,Y\rangle \\
\nonumber && -\langle d\varphi(X), \nabla_Y\tau(\varphi)\rangle-\langle
d\varphi(Y), \nabla_X\tau(\varphi)\rangle,
\end{eqnarray*}
and it satisfies the condition
\begin{equation}\label{eq:2-stress-condition}
\Div S_2=-\langle\tau_2(\varphi),d\varphi\rangle,
\end{equation}
thus conforming to the principle of a  stress-energy tensor for the
bienergy.\\

If  $\varphi:(M,g)\to (N,h)$ is an isometric immersion, then \eqref{eq:2-stress-condition} becomes
$$
\Div S_2=-\tau_2(\varphi)^{\top}.
$$
This means that isometric immersions with $\Div S_2=0$ correspond to immersions with va\-nishing tangent part of the corresponding bitension field.
The decomposition of the bitension field with respect to its normal and
tangent components was obtained with contributions of \cite{BMO13,C84,LM08,O02,O10} and for hypersurfaces it can be summarized in the following theorem.

\begin{theorem}\label{th: bih subm N}
Let  $\varphi:M^{n-1}\to N^{n}$ be an isometric immersion with mean curvature vector field $H=(f/(n-1))\,\eta$. Then, $\varphi$ is biharmonic if and only if the normal and the tangent components of $\tau_2(\varphi)$  vanish,  i.e. respectively

\begin{subequations}
\begin{equation}\label{eq: caract_bih_normal}
\Delta {f}+f |A|^2- f\ricci^N(\eta,\eta)=0
\end{equation}
and
\begin{eqnarray}\label{eq: caract_bih_tangent}
2 A(\grad f)+  f \grad f-2 f \ricci^N(\eta)^{\top}=0 \,\, ,
\end{eqnarray}
\end{subequations}
where $A$ is the shape operator and $\ricci^N(\eta)^{\top}$ is the tangent  component of the Ricci curvature of $N$ in the direction of the unit normal $\eta$ of $M$ in $N$.
\end{theorem}

Finally, from \eqref{eq: caract_bih_tangent}, an isometric immersion $\varphi:M^{n-1}\to N^{n}$  satisfies $\Div S_2=0$, i.e. it is biconservative, if and only if
$$
2 A(\grad f)+  f \grad f-2 f \ricci^N(\eta)^{\top}=0
$$
which is equation~\eqref{eq-def-biconservative-space} given in the introduction.

\section{$SO(p) \times SO(q)$-invariant immersions into Euclidean spaces}

In this section we carry out the differential geometric work which is necessary in order to study biconservative $SO(p+1) \times SO(q+1)$-invariant immersions into the Euclidean $\R^n$, $n=p+q+2$. More precisely, assuming the canonical splitting $\R^n=\R^{p+1}\times\R^{q+1}$, we shall study isometric immersions of the following type:
\begin{equation}\label{equivariantimmersion}
\left .
  \begin{array}{cccccccccc}
    \varphi_{p,q} &\colon& M=& \s^p &\times& \s^q&\times&(a,b)&\to &\R^{p+1}\times\R^{q+1}\\
    &&&&&&&&& \\
   &&&(w&,&z&,&s\,)&\longmapsto & \quad (x(s)\,w,\,y(s)\,z) \,\, , \\
  \end{array}
\right .
\end{equation}
where $(a,b)$ is a real interval which will be precised during the analysis and $x(s),\, y(s)$ are smooth positive functions. When it is clear from the context, we shall write $\varphi$ instead of $\varphi_{p,q}$.
We shall also assume that
\begin{equation}\label{sascissacurvilinea}
    \dot{x}^2+\dot{y}^2 =1 \,\, ,
\end{equation}
so that the induced metric on the domain in \eqref{equivariantimmersion} is given by:
\begin{equation}\label{inducedmetric}
    g=x^2(s)\,g_{\s^p}+y^2(s)\,g_{\s^q}+ds^2 \,\,,
\end{equation}
where $g_{\s^p}$ and $g_{\s^q}$ denote the Euclidean metrics of the unit spheres $\s^p$ and $\s^q$ respectively. We also note that the unit normal to $\varphi(M)$ can be conveniently written as
\begin{equation}\label{unitnormal}
    \eta = (-\,\dot{y}\,w,\,\dot{x}\,z) \,\, .
\end{equation}

Immersions of type \eqref{equivariantimmersion} are $G=SO(p+1)
\times SO(q+1)$-invariant and therefore we can work in the framework of
 equivariant differential geometry (see \cite{ER}, \cite{Hsiang}, \cite{MR}). In particular, the orbit space coincides with the flat Euclidean first quadrant
 \begin{equation}\label{orbitspace}
    Q=\R^{n}/G=\left \{(x,\,y) \in \, \R^2 \,\, : \,\, x,\,y \, \geq 0\,
\right \} \,\, .
 \end{equation}
We note that regular (i.e., corresponding to a point $(x,y)$ with both $x,\,y >0$) orbits are of the type $\s^p \times \s^q$. The orbit associated to the origin is a single point, while the other points on the $x$-axis (respectively, the $y$-axis) correspond to $\s^p$ (respectively, $\s^q$). We also note that, since \eqref{sascissacurvilinea} holds, it is often convenient to express quantities with respect to the angle $\alpha$ that the profile curve $\gamma(s)=(x(s),\,y(s))$ forms with the $x$-axis. In particular, we have:
\begin{equation}\label{angoloalfa}
    \left \{ \begin{array}{l}
              \dot{x}=\cos \alpha \\
              \dot{y}=\sin \alpha
            \end{array}
\right .
\end{equation}
and also, for future use,
\begin{equation}\label{angoloalfabis}
    \dot{\alpha}=\ddot{y}\,\dot{x}-\,\ddot{x}\,\dot{y} \,\,.
    \end{equation}
The property that an immersion of type \eqref{equivariantimmersion} is biconservative (respectively, biharmonic) is equivalent to the fact that $\gamma$ verifies an ODE (respectively, a system of ODE) in the orbit space.
More precisely, we prove the following result:
 \begin{proposition}\label{condizionedibiarmonicita} Let $\varphi$ be an immersion as in \eqref{equivariantimmersion} and let
 \begin{equation}\label{definizionedif}
    f=\dot{\alpha}+p\, \frac{\sin \alpha}{x}- q \,\frac{\cos \alpha}{y} \,.
 \end{equation}

 Then the tangential and normal parts of the bitension field $\tau_2(\varphi)$ vanish if, respectively,
 \begin{equation}\label{bitensionecomponentetangente}
    \dot{f}\, (f+2\,\dot{\alpha}) =0 
 \end{equation}
and
 \begin{equation}\label{bitensionecomponentenormale}
   \ddot{f}+\dot{f}  \left(p\,\frac{\dot{x}}{x}+q\,\frac{\dot{y}}{y}\right) -f\left(p\,\left(\frac{\dot{y}}{x}\right)^2+q\,\left(\frac{\dot{x}}{y}\right)^2+\dot{\alpha}^2\right)=0\,.
 \end{equation}
 \end{proposition}

 \begin{proof}
 We write down explicitly $\tau(\varphi)$. For this purpose, let $\{X_i\}_{i=1}^p$ and $\{Y_a\}_{a=1}^q$ be local orthonormal frames on $\s^p$ and $\s^q$ respectively and let $\partial_s=\partial/\partial s$ be the tangent vector field to $(a,b)$. Then the frame
 \begin{equation}\label{eq:def-frame}
 \left\{\frac{X_i}{x},\frac{Y_a}{y},\partial_s\right\}_{i=1,\ldots, p;\;a=1,\ldots,q}
\end{equation}
 is a local orthonormal frame on $M$ with respect to the induced metric \eqref{inducedmetric}.
 Now, using the definition of the pull-back connection and the Weingarten equation of $\s^p$ in $\R^{p+1}$  (respectively of $\s^q$ in $\R^{q+1}$) we obtain
 \begin{equation}\label{eq:calculation-tau-1}
 \nabla^{\varphi}_{\frac{X_i}{x}}\,d\varphi\left(\frac{X_i}{x}\right)= \nabla^{\R^{p+1}}_{X_i}X_i\,,\quad \nabla^{\varphi}_{\frac{Y_a}{y}}\,d\varphi\left(\frac{Y_a}{y}\right)= \nabla^{\R^{q+1}}_{Y_a}Y_a\,.
\end{equation}
 Moreover, defining the following vector fields on $\R^n$
 $$
 Z_1(\tilde w,\tilde z)=(\tilde w,0)\,,\quad Z_2(\tilde w,\tilde z)=(0,\tilde z)\,,
 $$
 we have
 \begin{equation}\label{eq:calculation-tau-2}
 \begin{aligned}
 \nabla^{\varphi}_{\partial_s} \,d\varphi\left(\partial_s\right)=&\nabla^{\varphi}_{\partial_s} (\dot{x}\, w, \dot{y}\, z)=\nabla^{\varphi}_{\partial_s} (\dot{x}\, w, 0)+\nabla^{\varphi}_{\partial_s} (0, \dot{y}\, z)\\
 =&\nabla^{\varphi}_{\partial_s}[\frac{\dot{x}}{x}\, (x\, w,0)]+\nabla^{\varphi}_{\partial_s}[\frac{\dot{y}}{y}\, (0 , y\, w)]=\nabla^{\varphi}_{\partial_s}[\frac{\dot{x}}{x}\, (Z_1\circ\varphi)]+\nabla^{\varphi}_{\partial_s}[\frac{\dot{y}}{y}\, (Z_2\circ\varphi)]\\
 =&\frac{\ddot{x}\,x-\dot{x}^2}{x^2}\,(x\, w,0)+\frac{\dot{x}}{x}\,\nabla^{\R^n}_{(\dot{x}\,w,\dot{y}\,z)} Z_1+
 \frac{\ddot{y}\,y-\dot{y}^2}{y^2}\,(0,y\, z)+\frac{\dot{y}}{y}\,\nabla^{\R^n}_{(\dot{x}\,w,\dot{y}\,z)} Z_2\\
 =&(\ddot{x}\,w,\ddot{y}\,z)=\dot{\alpha}\, \eta\,.
 \end{aligned}
\end{equation}
 To conclude the computation of the tension field, bearing in mind that the nonzero Christoffel symbols of the metric \eqref{inducedmetric} are
 $$
 \Gamma_{ij}^{k}=^{\s^p}\hspace{-1.4mm}\Gamma_{ij}^{k}\;,\qquad \Gamma_{ab}^{c}=^{\s^q}\hspace{-1.4mm}\Gamma_{ab}^{c}\;,
 $$
 $$
 \Gamma_{ij}^{\alpha}=-x\,\dot{x} \, (g_{\s^p})_{ij}\,,\quad \Gamma_{ab}^{\alpha}=-y\,\dot{y} \, (g_{\s^q})_{ab}\,,\quad \alpha=p+q+1\,,
 $$
 we can write
 \begin{equation}\label{eq:calculation-tau-3}
 \nabla^M_{\frac{X_i}{x}}\frac{X_i}{x}=\frac{1}{x^2}\, \nabla_{X_i}^{\s^p}{X_i}-\frac{\dot{x}}{x}\,\partial_s\,,\quad
 \nabla^M_{\frac{Y_a}{y}}\frac{Y_a}{a}=\frac{1}{y^2}\, \nabla_{Y_a}^{\s^q}{Y_a}-\frac{\dot{y}}{y}\,\partial_s\,.
\end{equation}
Finally, taking into account \eqref{eq:calculation-tau-1}--\eqref{eq:calculation-tau-3}, and that
$$
d\varphi(\nabla_{X_i}^{\s^p}{X_i})=x^2\, \nabla_{X_i}^{\s^p(x)}{X_i}\,,\quad
d\varphi(\nabla_{Y_a}^{\s^q}{Y_a})=y^2\,\nabla^{\s^q(y)}_{Y_a}Y_a
$$
where $\s^p(x)$ and $\s^q(y)$ are the spheres of radius $x$ and $y$ respectively, we obtain
$$
\begin{aligned}
\tau(\varphi)=& \sum_i \left\{  \nabla^{\varphi}_{\frac{X_i}{x}}\,d\varphi\left(\frac{X_i}{x}\right)-d\varphi\left(\nabla^M_{\frac{X_i}{x}}\frac{X_i}{x}\right)\right\}+\sum_a \left\{  \nabla^{\varphi}_{\frac{Y_a}{y}}\,d\varphi\left(\frac{Y_a}{y}\right)-d\varphi\left(\nabla^M_{\frac{Y_a}{y}}\frac{Y_a}{y}\right)\right\}\\
&+\nabla^{\varphi}_{\partial_s} \,d\varphi\left(\partial_s\right)-d\varphi\left(\nabla_{\partial_s}^M \partial_s\right)\\
=& \sum_i \left\{  \nabla^{\R^{p+1}}_{X_i} X_i-\nabla_{X_i}^{\s^p(x)}{X_i}\right\}+
\sum_a \left\{ \nabla^{\R^{q+1}}_{Y_a}Y_a-\nabla^{\s^q(y)}_{Y_a}Y_a\right\}+\left(p\frac{\dot{x}}{x}+q\frac{\dot{y}}{y}\right)d\varphi(\partial_s)+\dot{\alpha}\, \eta\\
=& -\frac{p}{x} w- \frac{q}{y} z+\left(p\frac{\dot{x}}{x}+q\frac{\dot{y}}{y}\right)d\varphi(\partial_s)+\dot{\alpha}\, \eta\,.
\end{aligned}
$$
Now, since $\varphi$ is an isometric immersion, $\tau(\varphi)=f\, \eta$ and,  by direct inspection,
\begin{equation}\label{eq:formula-per-f}
f=\dot{\alpha}+p\,\frac{\dot{y}}{x}-q\,\frac{\dot{x}}{y}\,.
\end{equation}

We now proceed to the computation of the shape operator of the isometric immersion $\varphi$. Since, for $X\in C(T\s^{p})$, we have (using the notation of \eqref{eq:calculation-tau-2})
$$
\begin{aligned}
\nabla^{\varphi}_X \eta=&\nabla^{\varphi}_X (-\dot{y}\, w,\dot{x}\, z)=\nabla^{\varphi}_X[-\frac{\dot{y}}{x}(x\, w,0)+\frac{\dot{x}}{y}(0,y\, z)]\\
=&\nabla^{\varphi}_X[-\frac{\dot{y}}{x}\,(Z_1\circ\varphi)]+\nabla^{\varphi}_X[\frac{\dot{x}}{y}\,(Z_2\circ\varphi)]\\
=&-\dot{y}\, X=d\varphi(-\frac{\dot{y}}{x}\,X)\,,
\end{aligned}
$$
we conclude that
$$
A(X)=\frac{\dot{y}}{x}\,X\,.
$$
Similarly, for $Y\in C(T\s^q)$, we obtain
$$
 A(Y)=-\frac{\dot{x}}{y}\,Y\,.
$$
Moreover, with analogous computations we find
$$
\nabla^{\varphi}_{\partial_s} \eta= d\varphi(-\dot{\alpha}\,\partial_s)\,,
$$
thus
$$
A(\partial_s)=\dot{\alpha}\,\partial_s\,.
$$
Then, with respect to the frame \eqref{eq:def-frame}, the matrix of the shape operator is
the diagonal matrix with entries in the diagonal:
\begin{equation}\label{eq-diagonal-shape-operator}
\overbrace{\frac{\dot{y}}{x} \cdots \frac{\dot{y}}{x}}^{p-\text{times}}\;\; \,\,\overbrace{-\frac{\dot{x}}{y}\cdots  -\frac{\dot{x}}{y}}^{q-\text{times}}\;\; \,\,\dot{\alpha}
\end{equation}
Note that, from \eqref{eq-diagonal-shape-operator}, we recover \eqref{eq:formula-per-f}, since $f=\trace A$.
To compute the tangential and normal parts of the bitension field $\tau_2(\varphi)$ we use Theorem~\ref{th: bih subm N}.
Since $\grad\,f=\dot{f}\, \partial_s$, from \eqref{eq: caract_bih_tangent} and the expression of the shape operator \eqref{eq-diagonal-shape-operator},
we immediately deduce that the tangential component of $\tau_2(\varphi)$ vanishes when \eqref{bitensionecomponentetangente} is satisfied. As for the normal part, we need
to compute the Laplacian of $f$. Using the orthonormal frame \eqref{eq:def-frame} and taking into account \eqref{eq:calculation-tau-3}, we find
$$
\begin{aligned}
-\Delta f=&\sum_i \left\{ \frac{1}{x} X_i \left( \frac{1}{x} X_i(f)\right)- \left(\nabla^M_{\frac{X_i}{x}}\frac{X_i}{x}\right) f\right\} + \sum_a \left\{ \frac{1}{y} Y_a \left( \frac{1}{y} Y_a(f)\right)- \left(\nabla^M_{\frac{Y_a}{y}}\frac{Y_a}{y}\right) f\right\}\\
&+\partial_s(\partial_s(f))-\left(\nabla^M_{\partial_s}\partial_s\right)f\\
=& p\,\frac{\dot{x}}{x}\,\dot{f}+q\,\frac{\dot{y}}{y}\,\dot{f}+\ddot{f}\,.
\end{aligned}
$$
Finally, since
$$
|A|^2=p\,\left(\frac{\dot{y}}{x}\right)^2+q\,\left(\frac{\dot{x}}{y}\right)^2+\dot{\alpha}^2
$$
we obtain that the normal part of $\tau_2(\varphi)$ vanishes when \eqref{bitensionecomponentenormale} is satisfied.
\end{proof}

 \begin{remark} We point out that the function $f$ in \eqref{definizionedif} coincides, up to a constant factor, with the mean curvature function. In particular, we recover immediately from \eqref{bitensionecomponentetangente} the well-known property, already announced  in the introduction, that a CMC immersion in $\R^n$ is biconservative. We shall prove below that there exist biconservative immersions of type \eqref{equivariantimmersion} which are not CMC: we shall call them \emph{proper} biconservative immersions.
 \end{remark}

\section{Proper $SO(p) \times SO(q)$-invariant biconservative immersions}\label{section-biconservative-existence}

According to \eqref{bitensionecomponentetangente}, an immersion of type \eqref{equivariantimmersion} is proper biconservative if $f$ is not constant and
\begin{equation*}
    f+2\,\dot{\alpha} =0\,\,.
\end{equation*}
Taking into account \eqref{definizionedif}, we see that the previous equation is equivalent to:
\begin{equation}\label{equazione-biconservative}
   3\, \dot{\alpha}+p\, \frac{\sin \alpha}{x}- q \,\frac{\cos \alpha}{y} = 0 \,\, .
\end{equation}

\begin{remark}\label{re:specialsolutionofsystembiconservative}
Since the only curves with $\dot{\alpha}=\cst\neq 0$ are arcs of circles parametrized by arc length, it is easy to check by direct inspection that \eqref{equazione-biconservative} does not admit any solution with $\dot{\alpha}=\cst\neq 0$.  We also observe that the only solution of \eqref{equazione-biconservative} with $\dot{\alpha}=0$ is the line  $y=\sqrt{(q/p)}\; x$ (parametrized by arc length). This solution corresponds to a minimal cone in $\R^n$. We conclude that an immersion of type \eqref{equivariantimmersion} is proper biconservative if and only if the profile curve $\gamma$ is a solution of \eqref{equazione-biconservative} with $\dot{\alpha}$ not identically zero.
\end{remark}

An immediate consequence of \eqref{equazione-biconservative} is the following result.

\begin{proposition}
Let $\varphi_{p,q}$ be an immersion of type \eqref{equivariantimmersion}, $\gamma(s)=(x(s),y(s))$ its profile curve and assume that  $p=3 p'$ and $q=3 q'$.  If $\varphi_{p,q}$ is proper biconservative, then $\varphi_{p',q'}$
is minimal. Conversely, if  $\varphi_{p',q'}$  is minimal, then $\varphi_{p,q}$ is either minimal or proper biconservative.
\end{proposition}
\begin{proof}
The assertion follows easily from the fact that if $p=3 p'$ and $q=3 q'$, then \eqref{equazione-biconservative}  becomes

$$
3\left( \dot{\alpha}+p'\, \frac{\sin \alpha}{x}- q' \,\frac{\cos \alpha}{y}\right) =3 f_{p',q'}= 0\,.
$$
\end{proof}
 In order to state our results concerning existence and qualitative behaviour of solutions of \eqref{equazione-biconservative}, we first carry out some preliminary analytical work. First, we introduce two quantities which play an important role in the study of solutions of \eqref{equazione-biconservative}:
\begin{lemma}\label{monotoniadiIeJ} Let
\begin{equation}\label{definizionediIeJ}
    I=y^{(q \slash 3)}\, \cos \alpha \qquad {\rm and} \qquad J=x^{(p \slash 3)}\, \sin \alpha \,\, .
\end{equation}
Then $I$ and $J$ are increasing along solutions of \eqref{equazione-biconservative}.
\end{lemma}
\begin{proof}We compute
\begin{eqnarray}\label{calcoloIpunto}
  \dot{I} &=& \frac{q}{3}\, y^{(q \slash 3)-1} \,\dot{y} \, \cos \alpha - y^{(q \slash 3)}\, \sin \alpha \, \dot{\alpha}\\ \nonumber
   &=&  \frac{p}{3}\, \frac {\sin^2 \alpha}{x} \,y^{(q \slash 3)} \, \geq \, 0 \,\, ,
\end{eqnarray}
where, in order to obtain the second equality, we have used \eqref{equazione-biconservative} and \eqref{angoloalfa}. Similarly, we compute
\begin{equation}\label{calcoloJpunto}
    \dot{J}=\, \frac{q}{3}\, \frac {\cos^2 \alpha}{y} \,x^{(p \slash 3)} \, \geq \, 0 \,\, .
\end{equation}

\end{proof}
Next, we observe that \eqref{equazione-biconservative} is invariant by homotheties: in other words, if $\gamma(s)$ is a solution of \eqref{equazione-biconservative}, so is $\gamma_c(s)=(1 \slash c)\, \gamma(cs)$, $\forall \, c\, \neq 0$ . This invariance suggests to study the qualitative behaviour of solutions in the $(\vartheta,\,\alpha)$-plane, where the angle $\vartheta$ is related to $x,\,y$ by means of the usual polar coordinate transformation:
\begin{equation}\label{angolotheta}
    \left \{ \begin{array}{l}
              x=r\,\cos \vartheta \\
              y=r\,\sin \vartheta \,\, ,\,\, 0 < \vartheta < (\pi \slash 2) \,\, .
            \end{array}
\right .
\end{equation}

\begin{lemma} Let us consider the following vector field
\begin{equation}\label{campodivettoripianoalfatheta-X}
X(\vartheta,\alpha)=
 3\,(\,\cos \vartheta \,\sin \vartheta\,\sin (\alpha-\vartheta)\,)\,\frac{\partial}{\partial \vartheta} +(\,p\,\sin \alpha\,\sin \vartheta-q\,\cos \alpha\,\cos \vartheta)\,\frac{\partial}{\partial \alpha}
\end{equation}
in the $(\vartheta,\,\alpha)$-plane. The solutions of \eqref{equazione-biconservative} correspond to the trajectories of $X(\vartheta,\alpha)$, $0 < \vartheta < (\pi \slash 2)$, that is they are solutions of the following first order differential system:
\begin{equation}\label{campodivettoripianoalfatheta}
    \left \{ \begin{array}{l}
              \dot{\vartheta}=3\,\sin \vartheta \, \cos \vartheta \, \sin (\alpha-\vartheta) \\
              \dot{\alpha}=q\,\cos \alpha \, \cos \vartheta \,-\, p\,\sin \alpha \, \sin \vartheta \,\, .
            \end{array}
\right .
\end{equation}
\end{lemma}
\begin{proof}We use the following equalities:
\begin{equation}\label{cambiovariabilealfatheta}
    {\rm (i)}\,\, \frac{d \alpha}{d \vartheta}=\, \dot{\alpha}\,\frac{d s}{d \vartheta} \qquad {\rm (ii)}\,\,  \frac{d  \vartheta}{ d s}= \, \frac{\sin (\alpha-\vartheta)}{r} \,\, .
\end{equation}
The equality \eqref{cambiovariabilealfatheta}(i) is obvious; as for \eqref{cambiovariabilealfatheta}(ii), we use \eqref{angoloalfa} and observe that differentiation in \eqref{angolotheta} yields:
\begin{equation}\label{differentiation}
    \left \{ \begin{array}{l}
              \cos \alpha \, ds \,(=dx)\,=\,\cos \vartheta \, dr - r\, \sin \vartheta \, d\vartheta \\
              \sin \alpha \, ds \,(=dy)\,=\,\sin \vartheta \, dr + r\, \cos \vartheta \, d\vartheta  \,\, .
            \end{array}
\right .
\end{equation}
Next, we multiply the first equation in \eqref{differentiation} by $ -\,\sin \vartheta$, the second equation in \eqref{differentiation} by $ \cos \vartheta$ and then we add them to obtain
\begin{eqnarray*}
  r \, d\vartheta &=& (\sin \alpha \, \cos \vartheta \,-\, \cos \alpha \, \sin \vartheta) \,ds\\ \nonumber
   &=& \sin (\alpha-\vartheta)\,ds \,\, ,
\end{eqnarray*}
 from which \eqref{cambiovariabilealfatheta}(ii) follows immediately. Next, using \eqref{angolotheta} and \eqref{cambiovariabilealfatheta}(i), (ii) we find that \eqref{equazione-biconservative} becomes
 \begin{equation*}
    3 \,\frac{d \alpha}{d \vartheta} \, \frac{\sin (\alpha-\vartheta)}{r}+p\,\frac{\sin \alpha}{r \,\cos \vartheta}-q\,\frac{\cos \alpha}{r\,\sin \vartheta}=0\,\, ,
 \end{equation*}
 which we rewrite as:
 \begin{equation}\label{campodivettoricomeunoforma}
   3\,(\,\cos \vartheta \,\sin \vartheta\,\sin (\alpha-\vartheta)\,)\,d\alpha +(\,p\,\sin \alpha\,\sin \vartheta-q\,\cos \alpha\,\cos \vartheta)\,d\vartheta =0 \,\,.
 \end{equation}
 Finally, \eqref{campodivettoripianoalfatheta} follows readily from \eqref{campodivettoricomeunoforma}.

\end{proof}
\begin{remark} We point out that to each trajectory of the vector field $X(\vartheta,\alpha)$ corresponds a family of homothetic solutions of \eqref{equazione-biconservative}. We also notice that, though solutions of \eqref{campodivettoripianoalfatheta} are defined for all $s \,\in \,\R$, some care is needed to go back from these curves in the $(\vartheta,\,\alpha)$-plane to solutions of \eqref{equazione-biconservative} in the orbit space $Q$. To make this statement more explicit, it is enough to examine more in detail the minimal cone introduced in Remark~\ref{re:specialsolutionofsystembiconservative}, which corresponds to:
\begin{equation}\label{puntocriticochiave}
    \vartheta (s) \equiv  \alpha_0\,\, ; \qquad \alpha (s) \equiv \alpha_0 \,\, ,
\end{equation}
where $\alpha_0=\arctan \sqrt {(q \slash p)}$ .
In the orbit space $Q$, this trajectory becomes  the half-line $\gamma(s)=((\cos \alpha_0 \, )\,s,\,(\sin \alpha_0 \, )\,s)$ which, at $s=0$, reaches the boundary of $Q$ at the origin.
\end{remark}
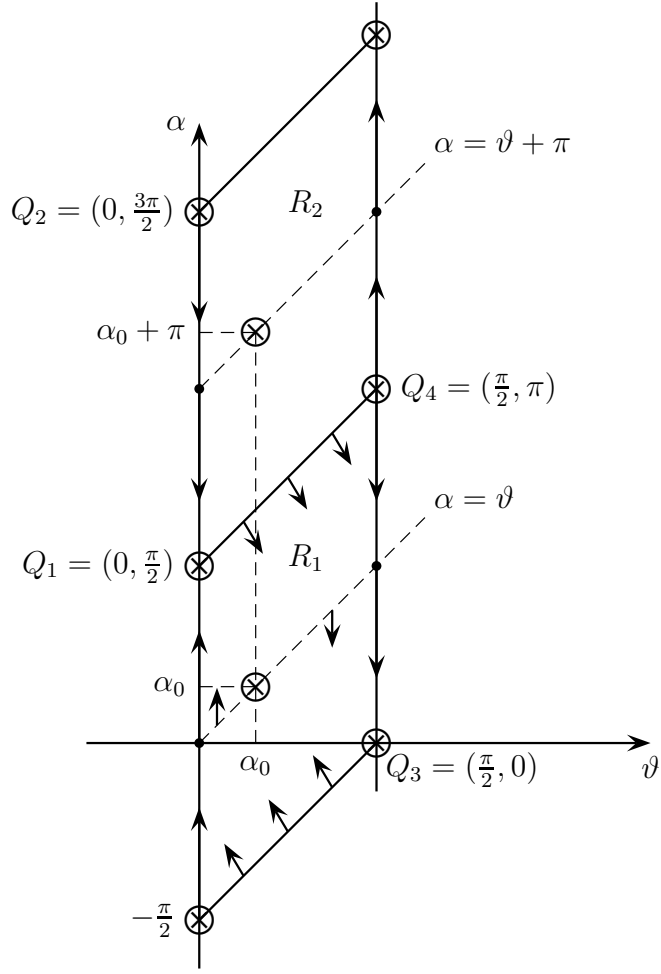
\begin{figure}[!htcb]
 \begin{center}
 \psset{unit=1.5cm,linewidth=.3mm,arrowscale=2}
\begin{pspicture}(-3,-2)(5,6.5)
\psline{->}(-1,0)(4,0)
\psline{->}(0,-2)(0,5.5)
\uput[-90](4,0){$\vartheta$}
\uput[180](0,5.5){$\alpha$}

\psline{->}(0,0)(0,1)
\psline{->}(0,-1.5708)(0,-.5708)
\psline{->}(0,3.14159)(0,2.14159)
\psline{->}(0,4.71239)(0,3.71239)

\psline{->}(1.5708,1.5708)(1.5708,.5708)
\psline{->}(1.5708,3.14159)(1.5708,2.14159)
\psline{->}(1.5708,3.14159)(1.5708,4.14159)
\psline{->}(1.5708,4.71239)(1.5708,5.71239)

\psline{}(1.5708,-0.429204)(1.5708,6.5708)

\psdot[dotstyle=otimes,dotscale=3](0,-1.5708)
\psdot[dotstyle=otimes,dotscale=3](0,1.5708)
\psdot[dotstyle=otimes,dotscale=3](1.5708,0)
\psdot[dotstyle=otimes,dotscale=3](1.5708,3.14159)
\psdot[dotstyle=otimes,dotscale=3](0,4.71239)
\psdot[dotstyle=otimes,dotscale=3](1.5708,6.28319)
\psdot[dotstyle=otimes,dotscale=3](.5,.5)
\psdot[dotstyle=otimes,dotscale=3](.5,3.64159)

\uput[180](-.1,-1.5708){$-\frac{\pi}{2}$}
\uput[180](-.1,1.5708){$Q_1=(0,\frac{\pi}{2})$}
\uput[180](-.1,4.71239){$Q_2=(0,\frac{3\pi}{2})$}
\uput[315](1.5708,0){$Q_3=(\frac{\pi}{2},0)$}
\uput[0](1.6708,3.14159){$Q_4=(\frac{\pi}{2},\pi)$}

\psdot[dotscale=1](0,3.14159)
\psdot[dotscale=1](1.5708,4.71239)
\psdot[dotscale=1](0,0)
\psdot[dotscale=1](1.5708,1.5708)

\psline[linestyle=dashed,linewidth=.16mm](0,3.14159)(2,5.14159)
\psline[linestyle=dashed,linewidth=.16mm](0,0)(2,2)
\uput[45](2,2){$\alpha=\vartheta$}
\uput[45](2,5.14159){$\alpha=\vartheta+\pi$}

\psline(0,-1.5708)(1.5708,0)
\psline(0,1.5708)(1.5708,3.14159)
\psline(0,4.71239)(1.5708,6.28319)
\psline[linestyle=dashed,linewidth=.16mm](.5,0)(.5,3.64159)
\psline[linestyle=dashed,linewidth=.16mm](0,.5)(.5,.5)
\psline[linestyle=dashed,linewidth=.16mm](0,3.64159)(.5,3.64159)

\uput[270](.5,0){$\alpha_0$}
\uput[180](0,.5){$\alpha_0$}
\uput[180](0,3.64159){$\alpha_0+\pi$}

\put(0.785398, 4.71239){$R_2$}

\put(0.785398,1.5708){$R_1$}

\psline{->}(0.392699, -1.1781)(0.2212, -0.892266)
\psline{->}(0.785398, -0.785398)(0.6139, -0.499567)
\psline{->}(1.1781, -0.392699)(1.0066, -0.106868)

\psline{->}(0.15708, 0.15708)(0.15708, 0.490413)
\psline{->}(1.1781, 1.1781)(1.1781, 0.844764)

\psline{->}(0.392699, 1.9635)(0.564198, 1.67766)
\psline{->}(0.785398, 2.35619)(0.956897, 2.07036)
\psline{->}(1.1781, 2.74889)(1.3496, 2.46306)
\end{pspicture}
     
  \end{center}
  \caption{The vector field $X(\vartheta,\alpha)$ in the $(\vartheta,\,\alpha)$-plane}\label{figuracampovettori}
\end{figure}
First, it is useful to observe that the vector field $X(\vartheta,\alpha)$ is defined on the region $0\leq \vartheta \leq (\pi/2)$. We also note that, if we consider a point with either $\vartheta=0$ or $\vartheta=\pi/2$, then
the integral curve passing through it is either a vertical segment or a point.

Now we proceed to the study of the qualitative behaviour of the trajectories of the vector field \eqref{campodivettoripianoalfatheta-X}. Since $0 \leq \vartheta \leq (\pi \slash 2)$ and the vector field has period $T=2 \pi$ with respect to the variable $\alpha$, it is enough to study trajectories in the region $R=R_1 \bigcup R_2$, where $R_1$ and $R_2$ are the two parallelograms of Figure \ref{figuracampovettori}: the analytical description of these regions in the $(\vartheta,\,\alpha)$-plane is
\begin{equation*}
    R_1= \left \{(\vartheta,\,\alpha)\; \colon 0 \leq \vartheta \leq \frac{\pi}{2}, \; \vartheta -\frac{\pi}{2}\leq \alpha \leq \vartheta +\frac{\pi}{2}\right \} \,\, ,
\end{equation*}
\begin{equation*}
    R_2= \left \{(\vartheta,\,\alpha)\; \colon 0 \leq \vartheta \leq \frac{\pi}{2}, \; \vartheta +\frac{\pi}{2}\leq \alpha \leq \vartheta +\frac{3\,\pi}{2}\right \} \,\, .
\end{equation*}

As a first step, in the following lemma we analyze the stationary points of the vector field.
\begin{lemma}\label{lemmapuntistazionari}Let $\alpha_0=\arctan \sqrt {(q \slash p)}$ . The stationary points (in $R$) of the differential system \eqref{campodivettoripianoalfatheta} are:
\begin{eqnarray}\label{puntistazionari}
  &&  P_0=(\alpha_0,\, \alpha_0)\,\, , \quad P_1=(\alpha_0,\,  \alpha_0+\pi)\,\, , \quad Q_1=\left(0,\, \frac{\pi}{2}\right )\,\, , \quad
    Q_2=\left(0,\, \frac{3\,\pi}{2}\right )\,\, , \\ \nonumber
    && Q_3=\left( \frac{\pi}{2},\,0 \right )\,\, , \quad Q_4=\left( \frac{\pi}{2},\, \pi \right )\,\, .
\end{eqnarray}
Moreover, $P_0$ is a spiral sink if $(p+q)\leq 17$ , while it is a nodal sink if $(p+q)\geq 18$. $P_1$ is a spiral source if $(p+q)\leq 17$, while it is a nodal source if $(p+q)\geq 18$. All the other stationary points are saddle points.
\end{lemma}
\begin{proof} The list \eqref{puntistazionari} of stationary points in $R$ can easily be obtained by direct inspection of \eqref{campodivettoripianoalfatheta}. Next, one has to evaluate the Jacobian matrix $J$ at each of the stationary points. At $P_0$ we find
\begin{equation}\label{jacobianainPzero}
    J(P_0)= \sin \alpha_0 \, \cos \alpha_0 \, \,\left[
             \begin{array}{cc}
               3 & -3 \\
               (p+q) & (p+q) \\
             \end{array}
           \right ]
\end{equation}
Now it is easy to check that, if $(p+q)\leq 17$, the eigenvalues of $J(P_0)$ are complex conjugate with negative real part (spiral sink), while, if $(p+q)\geq 18$, the eigenvalues of $J(P_0)$ are both real and negative (nodal sink). Similarly, a simple computation shows that $J(P_1)= - \, J(P_0)$: then it is again easy to conclude that, if $(p+q)\leq 17$, the eigenvalues of $J(P_1)$ are complex conjugate with positive real part (spiral source), while if $(p+q)\geq 18$ the eigenvalues of $J(P_1)$ are both real and positive (nodal source). As for the points $Q_i$, $i=1, \dots,4$, we have two real eigenvalues of opposite sign (saddle points).

\end{proof}
\begin{lemma}\label{permanenzatraiettorieinRuno} Let $(\vartheta(s),\alpha(s))$ be a trajectory of the differential system \eqref{campodivettoripianoalfatheta}. Suppose that, for some $s_0$, the trajectory is in $R_1$: $0<\vartheta(s_0)<(\pi \slash 2),\vartheta(s_0)-(\pi \slash 2)\leq\alpha(s_0)\leq\vartheta(s_0)+(\pi \slash 2)$. Then the trajectory remains in $R_1$ for all $s \geq s_0$.
\end{lemma}
\begin{proof} At both the upper and the lower edge of $R_1$ the vector field \eqref{campodivettoripianoalfatheta} points towards the interior of $R_1$ (see also Figure \ref{figuracampovettori}). Thus the trajectory cannot cross these bounds. On the other hand, the trajectory cannot reach neither $\vartheta=0$ nor $\vartheta=(\pi \slash 2)$, otherwise it would be contradicted the principle of uniqueness: more precisely, by uniqueness, any trajectory through a point of the type $\vartheta=0$ (respectively, $\vartheta=(\pi \slash 2)$) remains on this vertical straight line, a fact which makes our proof completed.

\end{proof}
\begin{lemma}\label{permanenzatraiettorieinRdue} Let $(\vartheta(s),\alpha(s))$ be a trajectory of the vector field \eqref{campodivettoripianoalfatheta}. Suppose that, for some $s_0$, the trajectory is in $R_2$: $0<\vartheta(s_0)<(\pi \slash 2),\vartheta(s_0)+(\pi \slash 2)\leq\alpha(s_0)\leq\vartheta(s_0)+(3\pi \slash 2)$. Then the trajectory is in $R_2$ for all $s \leq s_0$.
\end{lemma}
\begin{proof}The proof is analogous to that of Lemma \ref{permanenzatraiettorieinRuno}, so we omit the details.
\end{proof}
\begin{remark}\label{soluzioninonvannosugliassi} As a consequence of the arguments used in the proof of the previous two lemmata, we point out an important property of solutions of \eqref{equazione-biconservative}. Namely, if $(x(s),y(s))$ is a solution of \eqref{equazione-biconservative} in the orbit space $Q$, then it can reach the boundary of $Q$ at the origin only: more precisely, all the other boundary points are not allowed because the corresponding trajectory in the $(\vartheta,\,\alpha)$-plane would reach either the locus $\vartheta=0$ or the locus $\vartheta=(\pi \slash 2)$, a fact which is not possible, as explained above.
\end{remark}

\begin{figure}[!htcb]
 \begin{center}

\psset{unit=2cm,linewidth=.3mm,arrowscale=2}
\begin{pspicture}(-3,-2)(4,4)
\psline{->}(-1,0)(3,0)
\psline{->}(0,-2)(0,3.5)
\uput[-90](3,0){$\vartheta$}
\uput[180](0,3.5){$\alpha$}

\pscurve(0.,  1.5708 )
(0.15708  ,1.33754   )
(0.314159,  1.1173    )
(0.471239 , 0.918213  )
(0.628319,  0.742448  )
(0.785398,  0.588003   )
(0.942478 , 0.451059    )
(1.09956,  0.327455    )
(1.25664 , 0.213318    )
(1.41372,  0.1052    )
(1.5708  ,0. )

\psline{}(1.5708,-0.429204)(1.5708,3.5708)

\uput[180](-.1,-1.5708){$-\frac{\pi}{2}$}
\uput[180](-.1,1.5708){$\frac{\pi}{2}$}

\uput[315](1.5708,0){$\frac{\pi}{2}$}

\psdot[dotscale=1](0,0)
\psdot[dotscale=1](.685,.685)
\psdot[dotscale=1](1.5708,1.5708)

\psline[linestyle=dashed,linewidth=.16mm](0,0)(2,2)
\uput[45](2,2){$\alpha=\vartheta$}

\psline(0,-1.5708)(1.5708,0)
\psline(0,1.5708)(1.5708,3.14159)

\psline[linestyle=dashed,linewidth=.16mm](.685,0)(.685,.685)
\psline[linestyle=dashed,linewidth=.16mm](0,.685)(.685,.685)

\uput[270](.685,0){$\alpha_0$}
\uput[180](0,.685){$\alpha_0$}

\put(0.385398, -0.785398){$T_3$}

\put(0.785398,1.5708){$T_1$}
\put(1.3,.9){$T_2$}
\put(.1,.9){$T_4$}
\psline[linewidth=.1mm]{->}(2.,1.)(1.25664, 0.213318)
\uput[0](2,1){$\alpha=g(\vartheta)$}
\psdot[dotscale=1](1,1)
\uput[100](1,1){$w_1$}
\psdot[dotscale=1](.9,0.486604)
\uput[0](.9,0.486604){$w_2$}
\psdot[dotscale=1](0.53, 0.53)
\uput[180](0.53, 0.53){$w_3$}
\psdot[dotscale=1](0.59, 0.783236)
\uput[55](0.59, 0.783236){$w_4$}

\pscurve[curvature=0.3 0.3 0.3](1, 1)(.9,0.486604)(0.53, 0.53)(0.59, 0.783236)
\end{pspicture}

  \end{center}
  \caption{Non-existence of periodic orbits: case $p>q$.}\label{figuranoperiodichepmaggioreq}
\end{figure}
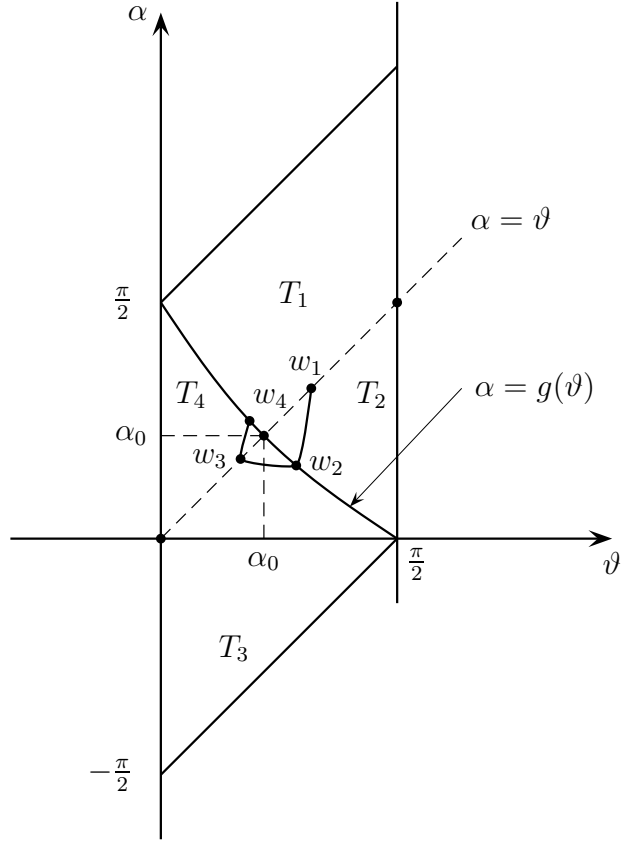

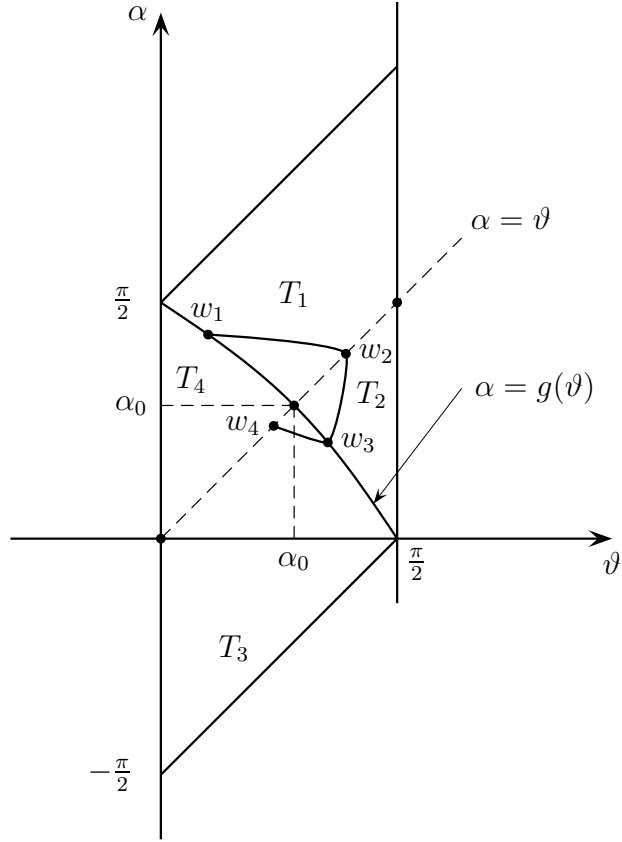
\begin{figure}[!htcb]
 \begin{center}

\psset{unit=2cm,linewidth=.3mm,arrowscale=2}
\begin{pspicture}(-3,-2)(4,4)
\psline{->}(-1,0)(3,0)
\psline{->}(0,-2)(0,3.5)
\uput[-90](3,0){$\vartheta$}
\uput[180](0,3.5){$\alpha$}

\psline{}(1.5708,-0.429204)(1.5708,3.5708)

\uput[180](-.1,-1.5708){$-\frac{\pi}{2}$}
\uput[180](-.1,1.5708){$\frac{\pi}{2}$}

\uput[315](1.5708,0){$\frac{\pi}{2}$}

\psdot[dotscale=1](0,0)
\psdot[dotscale=1](.886,.886)
\psdot[dotscale=1](1.5708,1.5708)

\psline[linestyle=dashed,linewidth=.16mm](0,0)(2,2)
\uput[45](2,2){$\alpha=\vartheta$}

\pscurve(0.  , 1.5708)
(0.15708 ,  1.4656   )
(0.314159 ,  1.35748  )
(0.471239 ,  1.24334  )
(0.628319 ,  1.11974  )
(0.785398  , 0.982794)
(0.942478 ,  0.828349 )
(1.09956 ,  0.652583  )
(1.25664 ,  0.4535   )
(1.41372 ,  0.233252 )
(1.5708  , 0.  )

\psline(0,-1.5708)(1.5708,0)
\psline(0,1.5708)(1.5708,3.14159)

\psline[linestyle=dashed,linewidth=.16mm](.886,0)(.886,.886)
\psline[linestyle=dashed,linewidth=.16mm](0,.886)(.886,.886)

\uput[270](.886,0){$\alpha_0$}
\uput[180](0,.886){$\alpha_0$}

\put(0.385398, -0.785398){$T_3$}

\put(0.785398,1.5708){$T_1$}
\put(1.3,.9){$T_2$}
\put(.1,1){$T_4$}
\psline[linewidth=.1mm]{->}(2.,1.)(1.41372, 0.233252)
\uput[0](2,1){$\alpha=g(\vartheta)$}

\psdot[dotscale=1](0.314159,1.35748)
\uput[90](0.314159,1.35748){$w_1$}
\psdot[dotscale=1](1.23,1.23)
\uput[0](1.23,1.23){$w_2$}
\psdot[dotscale=1](1.11, 0.640076)
\uput[0](1.11, 0.640076){$w_3$}
\psdot[dotscale=1](.75,.75)
\uput[180](.75,.75){$w_4$}

\pscurve[curvature=0.3 0.3 0.3](0.314159,1.35748)(1.23,1.23)(1.11,0.640076)(.75, .75)
\end{pspicture}
  \end{center}
  \caption{Non-existence of periodic orbits: case $p<q$.}\label{figuranoperiodichepminoreq}
\end{figure}

\begin{lemma}\label{nienteorbiteperiodicheinRuno}  Let $(\vartheta(s),\alpha(s))$ be a trajectory of the differential system \eqref{campodivettoripianoalfatheta}. Suppose that, for some $s_0$, the trajectory is in $R_1$, with $0<\vartheta(s_0)<(\pi \slash 2),\vartheta(s_0)-(\pi \slash 2)\leq\alpha(s_0)\leq\vartheta(s_0)+(\pi \slash 2)$. Then
\begin{equation}\label{ProprietalimiteinRuno}
    \lim_{s\rightarrow +\infty} (\vartheta(s),\alpha(s)) \,=\, P_0 \,\, .
\end{equation}
\end{lemma}
\begin{proof} The Poincar\'e-Bendixson theory (\cite{Hartman}), together with Lemmata \ref{lemmapuntistazionari}, \ref{permanenzatraiettorieinRuno}, tell us that the conclusion follows if we prove that the differential system \eqref{campodivettoripianoalfatheta} has no periodic orbit in $R_1$ (apart from the stationary point $P_0$). To this purpose, it is convenient to divide the region $R_1$ into four subregions $T_i$, $i=1,\ldots,4$, as in Figures \ref{figuranoperiodichepmaggioreq} and \ref{figuranoperiodichepminoreq}: this partition of $R_1$ is obtained by considering the curves $\alpha = \vartheta$ and $\alpha = g(\vartheta)$, where
\begin{equation*}
 g(\vartheta)=\cot ^{-1} \left ( \, \frac{p}{q} \, \tan \vartheta \,\right ) \,\, .
\end{equation*}
It is important to note that these are precisely the two curves where $\dot{\vartheta}=0$ and $\dot{\alpha}=0$ respectively. In particular, in the interior of this four subregions we have:
\begin{itemize}
  \item [(i)] $\dot{\vartheta}>0$ and $\dot{\alpha}<0$ in $T_1$;
  \item [(ii)] $\dot{\vartheta}<0$ and $\dot{\alpha}<0$ in $T_2$;
  \item [(iii)] $\dot{\vartheta}<0$ and $\dot{\alpha}>0$ in $T_3$;
  \item [(iv)] $\dot{\vartheta}>0$ and $\dot{\alpha}>0$ in $T_4$.
\end{itemize}
By the general theory (see \cite{Hartman}) we know that, if there exists a periodic orbit in $R_1$, it must enclose the equilibrium point $P_0$ (actually, this statement is also a simple consequence of our partition of the region $R_1$). Assume first that $p>q$, so that we are in the situation of Figure \ref{figuranoperiodichepmaggioreq}. Preliminarily, it is useful to observe that the function $g(\vartheta)$ is strictly decreasing and $g(\vartheta)=g^{-1}(\vartheta)$, a fact which is a consequence of the symmetry of the curve
$$
q\,\cos \alpha \, \cos \vartheta \,-\, p\,\sin \alpha \, \sin \vartheta =0
$$
with respect to the bisector line $\alpha=\vartheta$. Now, if there is a periodic orbit, it must cross the subdividing curves, as $s$ increases, at a sequence of points $w_i=(\vartheta_i,\alpha_i)$, $i=1,\ldots,4$, in such a way that, keeping into account the signs of $\dot{\vartheta},\,\dot{\alpha}$ inside the various subregions,
\begin{equation}\label{disequazionenonorbiteperiodiche}
    \alpha_4 < g(g(\vartheta_1))=\vartheta_1=\alpha_1 \,\, :
\end{equation}
but the inequality in \eqref{disequazionenonorbiteperiodiche}, together with the periodicity, contradicts the fact that $\dot{\alpha}<0$ in $T_1$. The case $p<q$ (illustrated in Figure \ref{figuranoperiodichepminoreq}) can be handled similarly, starting with $w_1$ on the boundary between $T_1$ and $T_4$ and deriving $\vartheta_4 > \vartheta_1$, a fact which contradicts $\dot{\vartheta}>0$ in $T_4$. The case $p=q$ (in which $g(\vartheta)=(\pi \slash 2)- \vartheta$) can be treated in either way, so the proof is ended.

\end{proof}
In an analogous way, we also have (the details of the proof are precisely as in Lemma \ref{nienteorbiteperiodicheinRuno}):
\begin{lemma}\label{nienteorbiteperiodicheinRdue}  Let $=(\vartheta(s),\alpha(s))$ be a trajectory of the differential system \eqref{campodivettoripianoalfatheta}. Suppose that, for some $s_0$, the trajectory is in $R_2$, with $0<\vartheta(s_0)<(\pi \slash 2),\vartheta(s_0)+(\pi \slash 2)\leq\alpha(s_0)\leq\vartheta(s_0)+(3\,\pi \slash 2)$. Then
\begin{equation}\label{ProprietalimiteinRdue}
    \lim_{s\rightarrow -\infty} (\vartheta(s),\alpha(s)) \,=\, P_1 \,\, .
\end{equation}
\end{lemma}
We are now in the right position to transfer all this material to the orbit space $Q$: our results are summarized in the following

\begin{theorem}\label{teoremabiconservative} There exists an infinite family of proper $SO(p+1)
\times SO(q+1)$-invariant biconservative immersions (cones) in $\R^n$ ($n=p+q+2$), of type \eqref{equivariantimmersion}. Their corresponding profile curves $\gamma(s)$ are defined on intervals of the type either (i) $[s_0,\,+\infty)$ or (ii) $(-\infty,\,s_0]$, with $\gamma(s_0)=(0,0)$ in both cases and $\gamma(s)$ in the interior of $Q$ if $s \neq s_0$. Moreover, in the case (i), as $s$ increases to $+\infty$, the curve $\gamma(s)$ tends asymptotically to the profile of the minimal cone, i.e., $(q \, \cos \alpha_0 \, x -p\,\sin \alpha_0 \, y)=0$ . Similarly, in the case (ii), as $s$ decreases to $-\infty$, the curve $\gamma(s)$ tends asymptotically to the profile of the minimal cone. In both cases: if $p+q \leq 17$, at infinity the profile curves $\gamma$ intersect the profile of the minimal cone at infinitely many points, while, if $p+q \geq 18$, at infinity the profile curves $\gamma$ do not intersect the profile of the minimal cone.
None of these hypersurfaces is complete.
\end{theorem}
\begin{proof} Let $\gamma(s)=(x(s),y(s))$ be a local solution of \eqref{equazione-biconservative} in the interior of $Q$. Let us first assume that, at some point $s^*$, the corresponding trajectory in the $(\vartheta,\,\alpha)$-plane is in $R_1$. Then we can assume that
\begin{equation*}
    - \frac{\pi}{2}< \alpha(s^*) < \pi \,\, .
\end{equation*}
We know, from Remark \ref{soluzioninonvannosugliassi}, that $\gamma(s)$ will be defined for all $s \geq s^*$ unless it reaches the origin $(0,0)$. We argue by contradiction: suppose that
\begin{equation*}
    - \frac{\pi}{2}< \alpha(s^*) < \frac{\pi}{2}\,\, :
\end{equation*}
then $I(s^*)>0$. Therefore the hypothesis that $\gamma$ reaches the origin, for some $s > s^*$, would contradict the fact that, according to Lemma \ref{monotoniadiIeJ}, $I(s)$ is increasing along solutions. Similarly, if
\begin{equation*}
    0< \alpha(s^*) < \pi\,\, ,
\end{equation*}
one uses the monotonicity of $J(s)$ to conclude that $\gamma$ cannot reach the origin. By way of summary, we conclude that $\gamma(s)$ is defined at least for all $s \geq s^*$. Next, the qualitative asymptotic behaviour of $\gamma(s)$ as $s$ tends to $+ \infty$ is an immediate consequence of the fact (see Lemma \ref{lemmapuntistazionari}) that $P_0$ is a spiral sink if $(p+q) \leq 17$, and a nodal sink if $(p+q) \geq 18$. At this stage, we have to investigate the qualitative behaviour of our solution for $s<s^*$: to summarize, we have only two possibilities:
$$
\,\,
$$
(A) The solution is defined for all $s<s^*$ ;

(B) The solution reaches the origin at some $s_0<s^*$ .
$$
\,\,
$$
In order to complete our analysis, it is enough to show that (A) is not possible, so that (B) holds. So, arguing again by contradiction, let us assume that (A) holds. Then necessarily (use Figure \ref{figuracampovettori} and Lemma \ref{nienteorbiteperiodicheinRdue}) the corresponding trajectory in the $(\vartheta,\alpha)$-plane must leave $R_1$ and tend to $P_1$ as $s$ decreases to $-\infty$. In particular, say near $-\infty$, there exists a point $\bar{s}$ at which $\pi < \alpha(\bar{s})< (3 \,\pi \slash 2)$, so that $I(\bar{s})<0$ and $J(\bar{s})<0$. Because $I$ and $J$ must become both positive moving along the solution in the sense of increasing values of $s$ (because, in the $(\vartheta,\,\alpha)$-plane, the trajectory tends to $P_0$) , we conclude that the solution must pass through the origin, a fact which makes (A) not possible and so confirms (B).

In a dual way, one completes the proof by studying the qualitative behaviour, as $s$ decreases to $-\infty$, of solutions with a point in $R_2$: since the arguments are the same as above, we omit the details.

\end{proof}
\begin{remark}\label{geometriadellesoluzioni} We have chosen the above formulation for Theorem \ref{teoremabiconservative} because we wanted to give a fairly complete portrait of solutions in $Q$ and of their counterparts in the $(\vartheta,\,\alpha)$-plane. However, we point out that, up to reparametrization, each of the profiles of the biconservative cones of Theorem \ref{teoremabiconservative} could be described by means of a curve $\gamma(s)$, $s\geq0$, with $\gamma(0)=(0,0)$ and $\gamma(s)$ in the interior of $Q$ for $s>0$ . By way of example, there is no geometric difference between the profiles $\gamma(s)=(\cos \alpha_0 \,s ,\,\sin \alpha_0 \,s)$ with $s \geq 0$ (i.e., the minimal cone $\alpha(s)\equiv \alpha_0$), and $\gamma(s)=(\cos (\alpha_0+\pi)\,s ,\,\sin (\alpha_0+\pi) \,s)=(-\,\cos \alpha_0\,s ,\,-\,\sin \alpha_0 \,s)$ with $s \leq 0$ (i.e., the same minimal cone represented as $\alpha(s)\equiv (\alpha_0+\pi)$).
\end{remark}

\section{$SO(p+1) \times SO(q+1)$-invariant biharmonic immersions}\label{biharmonic-immersion-non-existence}

In this section we prove that biharmonic $SO(p+1) \times SO(q+1)$-invariant immersions into the Euclidean space $\R^n$, $n=p+q+2$, are minimal. More precisely, we obtain the following:

\begin{theorem}\label{teo-non-existence-biharmonic}
Let  $\varphi : M= \s^p \times \s^q\times (a,b)\to\R^{p+1}\times\R^{q+1}$ be a $SO(p+1) \times SO(q+1)$-invariant biharmonic immersion, that is a map of type \eqref{equivariantimmersion} where $x(s)$ and $y(s)$
are solutions of the system of ODE:
\begin{equation}\label{eq-syst-main}
\begin{cases}
\dot{f}\left(3(\ddot{y}\,\dot{x}-\ddot{x}\,\dot{y})+p\, \dfrac{\dot{y}}{x}- q \,\dfrac{\dot{x}}{y}\right)=0\\
\\
 \ddot{f}+\,p\,\dfrac{\dot{x}}{x}\,\dot{f}\,+\, q\,\dfrac{\dot{y}}{y}\,\dot{f}-f\left(p\left(\dfrac{\dot{y}}{x}\right)^2+q\left(\dfrac{\dot{x}}{y}\right)^2+(\ddot{y}\,\dot{x}-\ddot{x}\,\dot{y})^2 \right)=0\,,
\end{cases}
\end{equation}
with
$$
f=(\ddot{y}\dot{x}-\ddot{x}\dot{y})+p\, \dfrac{\dot{y}}{x}- q \,\dfrac{\dot{x}}{y}\,.
$$
Then $\varphi$ is a minimal immersion.

\end{theorem}
\begin{proof}

It is sufficient to show that $\varphi$ is a CMC immersion. In fact,  biharmonic CMC immersions in
$\R^n$ are minimal.
Assume that $\varphi$ is not CMC, then there exists an open interval $I$ of $(a,b)$ where $\dot{f}(s)> 0$, for all $s\in I$.

From the first equation in \eqref{eq-syst-main}, multiplied by $\dot{x}$, it is easy to deduce that
\begin{equation}\label{eq-explicity2}
\ddot{y}= -\frac{\dot{x}}{3}\left(p\, \dfrac{\dot{y}}{x}- q \,\dfrac{\dot{x}}{y}\right)\,.
\end{equation}
In the same way, multiplying by $\dot{y}$, we have
\begin{equation}\label{eq-explicitx2}
\ddot{x}= \frac{\dot{y}}{3}\left(p\, \dfrac{\dot{y}}{x}- q \,\dfrac{\dot{x}}{y}\right)\,.
\end{equation}
Using \eqref{eq-explicity2} and \eqref{eq-explicitx2} the expression of $f$ becomes
\begin{equation}\label{eq-explicitf}
f=\frac{2}{3}\left( p\, \dfrac{\dot{y}}{x}- q \,\dfrac{\dot{x}}{y}\right)\,.
\end{equation}
Now, using \eqref{eq-explicity2}, \eqref{eq-explicitx2} and \eqref{eq-explicitf}, we find that the second equation of \eqref{eq-syst-main} takes the form
\begin{equation}\label{eq-normal-explicit}
A(x,y) \,\dot{x}^2\, \dot{y}
+B(x,y)\, \dot{x}\, \dot{y}^2
+C(x,y)\, \dot{x}
+D(x,y)\, \dot{y}=0\,,
\end{equation}
where
$$
\begin{cases}
A(x,y)=3\,p\,(3+2p)\,y^3-6 \,p\, q \, x^2\,y\\
B(x,y)=-3\,q\,(3+2q)\,x^3+6\, p\, q\, x\, y^2\\
C(x,y)=q^2\,(6+q)\,x^3+p\, q (p-3)\, x\, y^2\\
D(x,y)= -p^2\, (p+6) \,y^3-p\, \,q (q-3)\,x^2\,y\,.
\end{cases}
$$

For a fixed $s_0\in I$ we put $x_0=x(s_0)$. Since $\dot{x}^2+\dot{y}^2=1$, we can express $y$ as a function of $x$, $y=y(x)$, with  $x\in(x_0-\varepsilon,x_0+\varepsilon)$, and write
\begin{equation}\label{eq-non-ex-1}
\dot{y}=\frac{dy}{dx}\, \dot{x}\,.
\end{equation}
From $\dot{x}^2+\dot{y}^2=1$ we obtain
\begin{equation}\label{eq-non-ex-2}
\dot{x}^2=\dfrac{1}{1+\left(\dfrac{dy}{dx}\right)^2}\,.
\end{equation}
Deriving \eqref{eq-non-ex-1} with respect to $s$ an easy computation leads us to
\begin{equation}\label{eq-non-ex-3}
\ddot{y}=\dfrac{1}{\left(1+\left(\dfrac{dy}{dx}\right)^2\right)^2}\, \dfrac{d^2y}{dx^2}\,,
\end{equation}
that, together with \eqref{eq-explicity2}, gives
\begin{equation}\label{eq-non-ex-4}
\frac{d^2y}{dx^2}=\frac{1}{3} \left(
1+\left(\dfrac{dy}{dx}\right)^2\right)\left(\frac{q}{y}-\frac{p}{x}\dfrac{dy}{dx} \right) \,.
\end{equation}
Substituting \eqref{eq-non-ex-1} and \eqref{eq-non-ex-2} in \eqref{eq-normal-explicit} we obtain, up to a multiplicative factor $\dot{x}/\left(1+\left({dy}/{dx}\right)^2\right)$,
$$
D(x,y) \left(\dfrac{dy}{dx}\right)^3+(B(x,y)+C(x,y))\left(\dfrac{dy}{dx}\right)^2+(A(x,y)+D(x,y)) \left(\dfrac{dy}{dx}\right)+C(x,y)=0\,,
$$
which we rewrite as
\begin{equation}\label{eq-non-ex-5}
A_3(x,y) \left(\dfrac{dy}{dx}\right)^3+A_2(x,y)\left(\dfrac{dy}{dx}\right)^2+A_1(x,y) \left(\dfrac{dy}{dx}\right)+A_0(x,y)=0\,.
\end{equation}
Next, taking the derivative of \eqref{eq-non-ex-5} with respect to $x$ and bearing in mind \eqref{eq-non-ex-4}, we obtain
\begin{equation}\label{eq-non-ex-6}{
B_5 \left(\dfrac{dy}{dx}\right)^5+B_4 \left(\dfrac{dy}{dx}\right)^4+B_3 \left(\dfrac{dy}{dx}\right)^3+B_2\left(\dfrac{dy}{dx}\right)^2+B_1\left(\dfrac{dy}{dx}\right)+B_0=0
}
\end{equation}
where
$$
\begin{cases}
B_5(x,y)=3\,p^2\,q\,(q-3)\,x^2\,y^2+3\,p^3\,(p+6)\, y^4\\
B_4(x,y)=-p\,q\,\,(5\,q^2-6\,q-27)\,x^3\,y-p^2\,(5\,p\,q+9\,p+24\,q+54)\,x\,y^3\\
B_3(x,y)=2\,q^2\, (q^2-9)\, x^4 + 6\, p\, q\, (p\, q+6)\, x^2\, y^2 + p^2\, ( 4\, p^2+ 18\, p-9)\, y^4\\
B_2(x,y)=3\,q\, [p\,(-2\,q^2+q+3)+3\,(q^2-9)]\,x^3\,y+3\,p\,[-p^2\,(2\,q+3)-7\,p\,q+6\,q+27]\,x\,y^3\\
B_1(x,y)=2\,q^2\,(q^2-9)\, x^4 + 3\, p\, q\, [p\, (q+3)-12]\, x^2 y^2 + p^2\, (p^2-9)\, y^4\\
B_0(x,y)=-q^2\, [p\,(q+3)-9\,(q+6)]\,x^3\,y-p^2\,q\,(p-3)\,x\,y^3\,.
\end{cases}
$$

For any arbitrarily fixed $x_1\in(x_0-\varepsilon,x_0+\varepsilon)$, setting $y_1=y(x_1)$, \eqref{eq-non-ex-5} and \eqref{eq-non-ex-6}
can be thought as two polynomial equations in $dy/dx$, with coefficients given, respectively,  by $A_i(x_1,y_1),\, i=0,\ldots,3$ and $B_i(x_1,y_1),\, i=0,\ldots,5$,  which have the common solution $(dy/dx)(x_1)$.
Using standard arguments of algebraic geometry (\cite{HasVla95}), this implies that the resultant of the two polynomials is zero for any $x_1\in(x_0-\varepsilon,x_0+\varepsilon)$.  Now, since the coefficients $A_i(x,y)$ and $B_i(x,y)$ are homogeneous
polynomials of degree $3$ and $4$ respectively, it turns out that the resultant is a homogeneous polynomial of degree $27=3\cdot 5+ 4\cdot 3$. Then
the only real factors are of type $y-mx$ or of type $a^2\,x^2+b^2\,y^2$, and this implies that a common solution of \eqref{eq-non-ex-5} and \eqref{eq-non-ex-6} must be of the
form $y=mx$. Using Remark~\ref{re:specialsolutionofsystembiconservative} we know that the only solution of \eqref{eq-non-ex-5} of type $y=mx$
is $y=\sqrt{q/p}\;x$, which corresponds to the invariant minimal cone, a contradiction.

\end{proof}

\section{Proper $SO(p+1)$-invariant biconservative hypersurfaces}

In this section we investigate the existence of biconservative  $SO(p+1)$-invariant immersions into the Euclidean space $\R^n$, $n=p+2$. More precisely, we shall study isometric immersions of the following type:
\begin{equation}\label{equivariantimmersion-1}
\left .
  \begin{array}{cccccccc}
    \varphi &\colon& M=& \s^p&\times&(a,b)&\to &\R^{q+1}\\
    &&&&&&& \\
   &&&(w&,&s\,)&\longmapsto & \quad (x(s)\,w,\,y(s)) \,\, , \\
  \end{array}
\right .
\end{equation}
where $(a,b)$ is a real interval which will be precised during the analysis, $x(s)$ is a smooth positive function, while $y(s)$ is a smooth function with isolated zeros.
We shall also assume that
\begin{equation}\label{sascissacurvilinea-1}
    \dot{x}^2+\dot{y}^2 = 1 \,\, ,
\end{equation}
so that the induced metric on the domain in \eqref{equivariantimmersion-1} is given by:
\begin{equation}\label{inducedmetric-1}
    g=x^2(s)\,g_{\s^p}+ds^2 \,\,,
\end{equation}
The unit normal to $\varphi(M)$ can be written, in this case, as
\begin{equation}\label{unitnormal-1}
    \eta = (-\,\dot{y}\,w,\,\dot{x}) \,\, .
\end{equation}
Immersions of the type \eqref{equivariantimmersion-1} are $G=SO(p+1)$-invariant and the orbit space of the target coincides with the half plane
$$
    Q=\R^{n}/G=\left \{(x,\,y) \in \, \R^2 \,\, : \,\, x\, \geq 0\,
\right \} \,\, .
$$
We note that regular (i.e., corresponding to a point $(x,y)$ with $x >0$) orbits are of the type $\s^p$. The orbit associated to a point of the $y$-axis is a single point. Also in this case, since \eqref{sascissacurvilinea-1} holds, we express quantities with respect to the angle $\alpha$ that the profile curve $\gamma(s)=(x(s),\,y(s))$ forms with the $x$-axis. Thus, we have:
\begin{equation}\label{angoloalfa-1}
    \left \{ \begin{array}{l}
              \dot{x}=\cos \alpha \\
              \dot{y}=\sin \alpha
             \\
             \dot{\alpha}=\ddot{y}\,\dot{x}-\,\ddot{x}\,\dot{y}\,.
            \end{array}
\right .
\end{equation}

In this context we have the following analog of Proposition~\ref{condizionedibiarmonicita}:

 \begin{proposition}\label{condizionedibiarmonicita-1}
 Let $\varphi$ be an immersion as in \eqref{equivariantimmersion-1} and let
 \begin{equation}\label{definizionedif-1}
    f=\dot{\alpha}+p\, \frac{\dot{y}}{x} \,.
 \end{equation}

 Then the tangential and normal parts of the bitension field $\tau_2(\varphi)$ vanish if, respectively,
 \begin{equation}\label{bitensionecomponentetangente-1}
    \dot{f}\, (f+2\,\dot{\alpha}) =0 \, ,
 \end{equation}
and
 \begin{equation}\label{bitensionecomponentenormale-1}
   \ddot{f}+p\, \dot{f}\,\frac{\dot{x}}{x} -f\left(p\,\left(\frac{\dot{y}}{x}\right)^2+\dot{\alpha}^2\right)=0\,.
 \end{equation}
 \end{proposition}

\begin{remark}
We first point out that the statement in Proposition~\ref{condizionedibiarmonicita-1} are those of Proposition~\ref{condizionedibiarmonicita} with $q=0$.
Moreover,  the immersion \eqref{equivariantimmersion-1} has at most two different principal curvatures, namely
$$
\overbrace{\frac{\dot{y}}{x} \cdots \frac{\dot{y}}{x}}^{p-\text{times}}\;\; \,\,\dot{\alpha}\,.
$$
Since a biharmonic hypersurface with at most two different principal curvatures in $\R^n$ is minimal (see \cite{dimitric}), we conclude that there exists no proper biharmonic immersion of type \eqref{equivariantimmersion-1}.
\end{remark}

Now, we study proper biconservative hypersurfaces of type \eqref{equivariantimmersion-1}: that is, according to
\eqref{bitensionecomponentetangente-1}, we look for nonconstant functions $f$ such that
\begin{equation}\label{eq-biconservative-proper-1}
    f+2\,\dot{\alpha} =3\,\dot{\alpha}+p\,\frac{\sin\alpha}{x}=0\,\,.
\end{equation}
In this case, the analysis of the qualitative behaviour of solutions of \eqref{eq-biconservative-proper-1} is facilitated by the existence of the following prime integral (the proof is just a direct computation):
\begin{lemma}\label{monotoniadiIeJ-1} Let
\begin{equation}\label{definizionediIeJ-1}
    J=x^{(p \slash 3)}\, \sin \alpha \,\, .
\end{equation}
Then $J$  is constant along any solution of \eqref{eq-biconservative-proper-1}.
\end{lemma}

\begin{remark}\label{pro-reflecting-solution}
Let $(x(s),y(s),\alpha(s))$ be a solution of \eqref{eq-biconservative-proper-1} defined for $s\in I$. Then
it is easy to check the following properties of the solution:
\begin{enumerate}
\item the reflection across a horizontal line $y=y_0$, that is $(x(s),2y_0-y(s),-\alpha(s))$,  remains a solution (defined on $I$);
\item $(x(s+d),y(s+d),\alpha(s+d))$ remains a solution;
\item when $I=(-\varepsilon, \varepsilon)$, then $(x(-s),y(-s),\alpha(-s)+\pi)$ remains a solution.
\end{enumerate}

\end{remark}

Using the prime integral $J$, we can prove the main result of this section which can be stated as follows:

\begin{theorem}\label{teo-revolution-catenary}
There exists an infinite family of complete, proper $SO(p+1)$-invariant biconservative immersions in $\R^n$ ($n=p+2$) of type \eqref{equivariantimmersion-1}.
Their corresponding profile curves $\gamma(s)$ are of ``{catenary}'' type.
\end{theorem}

\begin{proof}
Let $\gamma(s)=(x(s),y(s))\,,\,s\in(a,b),$ be a solution of  \eqref{eq-biconservative-proper-1} with $x(s)>0$. Then, from Lemma~\ref{monotoniadiIeJ-1}, along $\gamma$
we have
$$
J=x^{(p \slash 3)}\, \sin\alpha=C=\cst\,.
$$
If $C=0$ then we must have $\sin\alpha=\dot{y}=0$ and this would imply that the solution has $f=\cst$. Thus we can assume that $C>0$.
Then we have $\sin\alpha=C\, x^{-p/3}$ and $\cos\alpha=\sqrt{1-C^2\, x^{-2p/3}}$, so that
\begin{equation}\label{eq-integralyx}
\frac{dy}{dx}=\frac{C\, x^{p/3}}{\sqrt{x^{2p/3}-C^2}}\,\,.
\end{equation}
From  \eqref{eq-integralyx} we obtain a local solution $y=y(x)$ of \eqref{eq-biconservative-proper-1} defined for $x\in(\sqrt[p]{C^3},\,+\infty)$ and when
$x$ tends to $\sqrt[p]{C^3}$ the curve becomes parallel to the $y$-axes. Moreover, $dy/dx>0$, which means that $y(x)$ is strictly increasing and, finally,
$\lim_{x\to +\infty} dy/dx=C$ (see Figure~\ref{fir-rotational1} (a)). Since the length of the curve $(x,y(x))$ is infinite, its reparametrization by
arc length is defined on $(s_0,\,+\infty)$.
Therefore, according to Remark~\ref{pro-reflecting-solution}, we can consider the solution $\tilde{\gamma}$ of \eqref{eq-biconservative-proper-1} defined on $(-\infty, +\infty)\setminus\{0\}$.
By uniqueness (considering the solution $\gamma$ of \eqref{eq-biconservative-proper-1} determined by the initial conditions $(x_0=\sqrt[p]{C^3},y_0,\alpha_0=\pi/2)$) we can extend $\tilde{\gamma}$ to a solution defined in $(-\infty,\,+ \infty)$. 
These curves are of ``catenary'' type, as shown in Figure~\ref{fir-rotational1} (b).

\end{proof}

\begin{figure}[!htcb]
 \begin{center}
 
\psset{unit=1cm,linewidth=.3mm,arrowscale=2}
\begin{pspicture}(0,-2)(4,4)
\psline{->}(0,0)(4,0)
\psline{->}(0,-2)(0,4)
\uput[-90](4,0){$x$}
\uput[180](0,4){$y$}

\psdot[dotscale=1](1,1)

\pscurve(1., 1.) (1.03141, 1.25)(1.12763, 1.5)(1.29468, 1.75)(1.54308,
   2.)(1.88842, 2.25) (2.35241, 2.5) (2.96419, 2.75) (3.7622,
  3.)

\psline[linestyle=dashed,linewidth=.16mm](1,0)(1,1)
\psline[linestyle=dashed,linewidth=.16mm](0,1)(1,1)

\uput[270](1,0){$\sqrt[p]{C^3}$}
\uput[180](0,1){$y_0$}

\uput[270](2,-2){(a)}

\end{pspicture}\hspace{20mm}
\begin{pspicture}(0,-2)(4,4)
\psline{->}(0,0)(4,0)
\psline{->}(0,-2)(0,4)
\uput[-90](4,0){$x$}
\uput[180](0,4){$y$}

\psdot[dotscale=1](1,1)

\pscurve(1., 1.) (1.03141, 1.25)(1.12763, 1.5)(1.29468, 1.75)(1.54308,
   2.)(1.88842, 2.25) (2.35241, 2.5) (2.96419, 2.75) (3.7622, 3.)

 \pscurve[linestyle=dashed,linewidth=.16mm](3.7622, -1.) (2.96419, -0.75) (2.35241, -0.5) (1.88842, -0.25)
(1.54308, 0.) (1.29468, 0.25) (1.12763, 0.5) (1.03141, 0.75) (1., 1.)

\psline[linestyle=dashed,linewidth=.16mm](1,0)(1,1)
\psline[linestyle=dashed,linewidth=.16mm](0,1)(1,1)
\psline[linestyle=dashed,linewidth=.16mm](0,1)(4,1)
\uput[270](1,0){$\sqrt[p]{C^3}$}
\uput[180](0,1){$y_0$}

\uput[270](2,-2){(b)}
\end{pspicture}

  \end{center}
  \caption{}\label{fir-rotational1}
\end{figure}

\end{document}